\documentclass[11pt,reqno]{amsart}
\usepackage[colorlinks=true,allcolors=blue,backref=page]{hyperref}
\usepackage{color}
\usepackage{booktabs}
\usepackage{longtable}
\usepackage{array}
\usepackage{amsmath, amssymb, amsthm}
\usepackage{aliascnt}

\usepackage{mathrsfs}
\usepackage{mathtools}
\usepackage{multicol}

\usepackage[noabbrev,capitalize,nameinlink]{cleveref}
\crefname{equation}{}{}
\usepackage{fullpage}
\usepackage[noadjust]{cite}
\usepackage{graphics}
\usepackage{pifont}
\usepackage{tikz}
\usepackage{tikz-cd}
\usepackage{bbm}
\usepackage[T1]{fontenc}

\usetikzlibrary{arrows.meta}

\usepackage{environ}
\usepackage{framed}
\usepackage{url}
\usepackage[linesnumbered,ruled,vlined]{algorithm2e}
\usepackage[noend]{algpseudocode}
\usepackage[labelfont=bf]{caption}
\usepackage{cite}
\usepackage{framed}
\usepackage[framemethod=tikz]{mdframed}
\usepackage{appendix}
\usepackage{graphicx}
\usepackage[textsize=tiny]{todonotes}
\usepackage{tcolorbox}
\usepackage{enumerate}
\usepackage[shortlabels]{enumitem}
\allowdisplaybreaks[1]

\apptocmd{\sloppy}{\hbadness 10000\relax}{}{} % magic BibTeX spacing fix

\crefname{algocf}{Algorithm}{Algorithms}

\crefname{equation}{}{} %remove ``Equation''
\crefname{conjecture}{Conjecture}{Conjectures} %add ``Conjecture''
\AtBeginEnvironment{appendices}{\crefalias{section}{appendix}} %appendices

\crefformat{enumi}{#2#1#3}
\crefrangeformat{enumi}{#3#1#4 to~#5#2#6}
\crefmultiformat{enumi}{#2#1#3}%
{ and~#2#1#3}{, #2#1#3}{ and~#2#1#3}

\usepackage[final,color]{showkeys} % final disables visible key tags

\colorlet{refkey}{orange!20}
\colorlet{labelkey}{blue!30}

\crefname{algocf}{Algorithm}{Algorithms}

\numberwithin{equation}{section}
\newtheorem{theorem}{Theorem}[section]
\crefname{theorem}{Theorem}{Theorems}
\Crefname{theorem}{Theorem}{Theorems}

\newaliascnt{proposition}{theorem}
\newtheorem{proposition}[proposition]{Proposition}
\aliascntresetthe{proposition}
\crefname{proposition}{Proposition}{Propositions}
\Crefname{proposition}{Proposition}{Propositions}

\newaliascnt{lemma}{theorem}
\newtheorem{lemma}[lemma]{Lemma}
\aliascntresetthe{lemma}
\crefname{lemma}{Lemma}{Lemmas}
\Crefname{lemma}{Lemma}{Lemmas}

\newaliascnt{claim}{theorem}

\aliascntresetthe{claim}
\crefname{claim}{Claim}{Claims}
\Crefname{claim}{Claim}{Claims}

\newaliascnt{dichotomy}{theorem}

\aliascntresetthe{dichotomy}
\crefname{dichotomy}{Dichotomy}{Dichotomies}
\Crefname{dichotomy}{Dichotomy}{Dichotomies}

\newaliascnt{corollary}{theorem}

\aliascntresetthe{corollary}
\crefname{corollary}{Corollary}{Corollaries}
\Crefname{corollary}{Corollary}{Corollaries}

\newaliascnt{conjecture}{theorem}

\aliascntresetthe{conjecture}
\crefname{conjecture}{Conjecture}{Conjectures}
\Crefname{conjecture}{Conjecture}{Conjectures}

\newtheorem*{question*}{Question}

\newaliascnt{fact}{theorem}

\aliascntresetthe{fact}
\crefname{fact}{Fact}{Facts}
\Crefname{fact}{Fact}{Facts}

\theoremstyle{definition}
\newaliascnt{definition}{theorem}

\aliascntresetthe{definition}
\crefname{definition}{Definition}{Definitions}
\Crefname{definition}{Definition}{Definitions}

\newaliascnt{problem}{theorem}

\aliascntresetthe{problem}
\crefname{problem}{Problem}{Problems}
\Crefname{problem}{Problem}{Problems}

\newaliascnt{question}{theorem}

\aliascntresetthe{question}
\crefname{question}{Question}{Questions}
\Crefname{question}{Question}{Questions}

\newtheorem*{definition*}{Definition}

\newaliascnt{example}{theorem}

\aliascntresetthe{example}
\crefname{example}{Example}{Examples}
\Crefname{example}{Example}{Examples}

\newaliascnt{setup}{theorem}

\aliascntresetthe{setup}
\crefname{setup}{Setup}{Setups}
\Crefname{setup}{Setup}{Setups}

\theoremstyle{remark}
\newtheorem*{remark}{Remark}

\newcommand{\mb}{\mathbb}

\newcommand{\eps}{\varepsilon}

\newcommand{\mc}{\mathcal}

\newcommand{\on}{\operatorname}

\newcommand{\R}{\mathbb{R}}
\newcommand{\C}{\mathbb{C}}
\newcommand{\Z}{\mathbb{Z}}
\newcommand{\Q}{\mathbb{Q}}

\newcommand{\Disc}{\on{Disc}}
\renewcommand{\O}{\mathcal{O}}
\newcommand{\covol}{\on{covol}}

\title{Remarks on the disproof of the unit distance conjecture}
\author{Noga Alon}
\email{nalon@math.princeton.edu}

\author{Thomas F. Bloom}
\email{thomas.bloom@manchester.ac.uk}

\author{W. T. Gowers}
\email{w.t.gowers@dpmms.cam.ac.uk}

\author{Daniel Litt}
\email{daniel.litt@utoronto.ca}

\author{Will Sawin}
\email{wsawin@math.princeton.edu}

\author{Arul Shankar}
\email{ashankar@math.toronto.edu}

\author{Jacob Tsimerman}
\email{jacobt@math.toronto.edu}

\author{Victor Wang}
\email{vywang@as.edu.tw}

\author{Melanie Matchett Wood}
\email{mmwood@math.harvard.edu}

\begin{document}
\begin{abstract}
We present a short, digested, human-verified version of the
recent OpenAI-generated counterexample to the Erd\H{o}s unit distance conjecture,
and a sequence of reflections on it.
The argument relies crucially on ideas that may,
at least in retrospect,
be attributed to
Ellenberg-Venkatesh,
Golod-Shafarevich,
and Hajir-Maire-Ramakrishna.
\end{abstract}

\maketitle

\setcounter{tocdepth}{1}
{\tiny
\begin{multicols}{3}
\tableofcontents
\end{multicols}
}
\setcounter{tocdepth}{3}

\section{Introduction}

The following result is due to an internal model at OpenAI.
The proof we give in these remarks is a human-digested, somewhat simplified, and somewhat generalized version of the AI proof.\footnote{Throughout this document,
phrases such as ``AI proof'' or ``GPT proof'' all refer to the same file,
which was first mathematically generated in one shot by an internal model at OpenAI,
and then expositionally refined through human interactions with Codex.}

\begin{theorem}\label{thm:main}
There exists $\eps>0$ such that the following holds. There exists a sequence of point sets $\mc{P}_i$ in $\mb{R}^2$ such that $|\mc{P}_i|\to \infty$ and the number of unit distances in $\mc{P}_i$ is at least $|\mc{P}_i|^{1+\eps}$ for all $i$.
\end{theorem}

\subsection{History of the problem}
The unit distance problem was originally raised in work of Erd\H{o}s \cite{Erdos1946SetsDistances}. In his original work, he noted that a set of $n$ points may have at most $O(n^{3/2})$ unit distances via noting that the unit distance graph cannot contain a $K_{2,3}$ (two unit circles can only intersect in at most $2$ points) while using a $\sqrt{n}\times \sqrt{n}$ grid to show that a set of $n$ points may have $n^{1+\Omega(1/\log\log n)}$ unit distances. 
The best current upper bound, $O(n^{4/3})$,
is due to Spencer, Szemer{\'e}di, and Trotter \cite{SpencerSzemerediTrotter1984UnitDistances}.
An upper bound of $n^{1+o(1)}$ was conjectured by Erd\H{o}s;
see \cite{BloomErdosProblem90} for details.

For more background, the reader may consult the introduction of
Alon, Buci\'c, and Sauermann \cite{Alon2025Typical},
a recent paper on the unit distance and distinct distances problems
for generic norms, which are shown to behave rather differently than the usual Euclidean norm on $\R^2$.
Indeed,
Alon has remarked that
% quoting Noga's email
problems in discrete geometry like this one are
closely connected to deep questions in real algebraic geometry and number
theory, as they deal with common roots of natural sets of real
polynomials.
% end quote
A similar insight appears to be part of the
Chain-of-Thought (CoT)
of the AI proof of Theorem~\ref{thm:main}:
``\dots in principle all extremal examples can be taken algebraic.  But the degree and height of that algebraic realization can be enormous\dots
Maybe that enormous degree is not just an annoyance but a source of possible counterexamples.  Number fields deserve a closer look.''

The idea of trying to use number fields
to construct counterexamples is not altogether new,
but there are subtleties to making it work
(discussed in some of the reflections in this note),
especially considering that perhaps most experts believed the conjecture $n^{1+o(1)}$ to be true.

\subsection{Sketch of the proof}

One starting point of the proof is to construct a large set $U$ of magnitude-$1$
algebraic numbers of bounded denominator $D$ in a number field $K$, in terms of
the class number $h(K)$ and
the splitting behavior of various prime ideals.
This will be done in Lemma~\ref{pigeons}
using the pigeonhole principle,
along the lines of an idea of
Michel, Soundararajan,
Ellenberg, and Venkatesh
recorded in \cite{EllenbergVenkatesh2007Reflection}.

It turns out to be convenient to assume $K$ is a CM field.
A key reason for this is that an element of a CM field has absolute value $1$
in some embedding if and only if it
has absolute value $1$ in all embeddings.
This will be convenient for ensuring that many elements of $U$
appear as differences $x-y$ of elements $x,y\in W$ for a nice bounded window $W$ of $D^{-1}\O_K$, to be discussed below.

The bounded window is constructed using either the geometry of numbers,
or an averaging (unfolding) argument.
This step requires embedding the ring of algebraic integers $\O_K$ of $K$ as a lattice inside $K\otimes \R$.
Equivalently, for a more analytically-minded reader,
it requires working with not just a single absolute value on $K$,
but the sup-norm over all conjugate absolute values on $K$.

The set of unit distance pairs created in $W$
decreases as the root discriminant of $K$ increases,
and thus it is advantageous to let $K$
be a finite layer of an infinite class field tower $M$ of Golod-Shafarevich type,
with $[K:\Q]\to \infty$, since such $K$ have bounded root discriminant.
The simplest way to make $U$ large in Lemma~\ref{pigeons}
is to ensure that this tower $M$ has at least one split rational prime $q$.
This fixed prime $q$ will split into many primes in $K$
as $[K:\Q]\to \infty$, and when used appropriately this drowns out
the main enemies, the class number $h(K)$ and discriminant $\Disc K$, which are relatively mild thanks to the controlled ramification in $M$.
Fixing an embedding $K\hookrightarrow \C$,
we obtain the desired sequence of point sets $\mathcal{P}_K := W \hookrightarrow \C=\R^2$,
indexed by fields $K$.

The AI argument originally required a suitably large finite number of split primes.
Construction of such Golod-Shafarevich towers is, however, not the novel part of the argument.
The construction of Golod-Shafarevich towers
with not just a single split prime,
% but infinitely many primes with bounded inertia degree,
but infinitely many split primes,
already appears in
the literature \cite{hajir2021cutting}
and has been used for other applications.

\subsection{Context for the proof}

The original grid construction can be thought of as an application of
Lemma~\ref{pigeons} to the CM field $K = \Q(i)$.
Counting magnitude-$1$ algebraic numbers
of bounded Weil height
in a fixed CM field $K$
has also been done before;
see, for instance, \cite{akhtari2025effective,Batyrev1998Manin}.
A novel ingredient of the AI argument is to take $[K:\Q]\to \infty$.
In classical Diophantine terms,
if $$r_{2,F}(\alpha) := |\{(x,y)\in \O_F^2: x^2+y^2 = \alpha\}|$$
denotes the number of ways to write $\alpha\in \O_F$ as a sum of two squares,
where $F$ is a totally real number field, then one corollary of the argument is that
there exist rational integers $D\ge 1$
such that $r_{2,F}(4D^2)$ grows exponentially in $[F:\Q]$
along an infinite tower of fields $F$.
Similar phenomena occur in the large-$q$ aspect
of analytic number theory over function fields $\mathbb{F}_q(t)$,
with $\log q$ being roughly analogous to $[F:\Q]$.
In contrast, $r_{2,F}(4D^2) \le O_{F,\eps}(D^\eps)$ for any fixed $F$ and $\eps>0$.

Infinite sequences of number fields have previously led to interesting combinatorial constructions.
Relatively recent examples include Venkatesh's sphere-packing construction \cite{Venkatesh2013SpherePackings}, and the work of Kopp et al.~on
equiangular lines, ray class fields, and the Stark conjectures \cite{Kopp2021SICPOVMStark,ApplebyFlammiaMcConnellYard2020RayClassFields}.

The fact that infinite class field towers provide lattices with strong asymptotic properties has been applied before in a number of nearby settings.  
Lenstra~\cite{Lenstra1986Codes} used such infinite class field towers to construct codes and study their asymptotic properties.
Litsyn and Tsfasman \cite[Section 7]{LitsynTsfasman1987Constructive} introduced lattice sphere packings using lattices from infinite class field towers.  This led to many further papers in coding theory studying asymptotic properties of these codes---some examples include \cite{VehkalahtiLuzzi2015NumberField,Maire2018,LuzziVehkalahtiLing2018AlmostUniversal,Guruswami2003Constructions, Liang2024QaryGV}.  See also \cite{Tsfasman1991GlobalFields} for a nice exposition of these lattices in the context of coding theory and sphere packing.
Sphere packing and coding theory are inherently high dimensional problems, and so the idea to use points from a high dimensional lattice is natural.
(To be clear, $\O_K$ embeds into $\C$ as an abelian group,
which is used to extract the final point sets for Theorem~\ref{thm:main}.
But it does not satisfy the discreteness axiom of a lattice if $[K:\Q] \ge 3$.)

Belolipetsky--Lubotzky
\cite{belolipetsky2012manifolds}
is an example of class field towers used to count lattices inside a fixed Lie group.

Ellenberg and Venkatesh \cite[Lemma 2.3]{EllenbergVenkatesh2007Reflection} famously used small split primes to bound the $\ell$-torsion in the class group of a number field.  They consider  $P_i^\ell$ for various split primes $P_i$ and use the prigeonhole principle to get two ideals representing the same element of the class group, and then use that ratio to produce an element of the number field which has small archimedian valuations and also bounded denominator.   At this level of description that is similar to the the strategy used here in Lemma~\ref{pigeons}, but many of the details and parameters of interest in the application are rather different, e.g. \cite{EllenbergVenkatesh2007Reflection} only needed one element of the number field and the precise size of the denominator was relevant, whereas the application of Lemma~\ref{pigeons} is to produce many elements and the size of the denominator is not relevant to Theorem~\ref{thm:main} (but may be relevant in optimizations).

In Proposition~\ref{HMRmethod1mod4} we will briefly review
the Frobenius-cutting argument of \cite{hajir2021cutting}
and how it can be applied to our situation to obtain infinitely many split primes.
A weaker version of this argument, cutting out relations only at a fixed finite depth,
was used in the original GPT paper.

\subsection{Organization of the paper}
In \cref{sec:proof} we give a complete proof of \cref{thm:main}. After \cref{sec:proof}, we collect a sequence of reflections on the proof and comments on the AI solution.

\subsection*{Acknowledgements}

We thank 
Peter Sarnak and Kannan Soundararajan
for helpful discussions.
NA is supported by NSF grant DMS-2154082. TB is a Royal Society University Research Fellow. WTG is a Professor at the Coll\`ege de France and a Research Professor at the University of Cambridge. His work in automatic theorem proving is supported by the Astera Foundation. DL is supported by an NSERC grant, “Anabelian methods in arithmetic and algebraic geometry,” a Sloan research fellowship, a Connaught New Researcher Award, and an Ontario Early Researcher Award. WS is supported by NSF grant DMS-2502029 and was a Sloan Research Fellow. AS is supported by an NSERC discovery grant. JT was partially supported by the Institute for Advanced Study, and the Charles Simonyi Endowment. VW is supported by NSTC grant 114-2115-M-001-010-MY2. MMW is supported by a
Packard Fellowship for Science and Engineering,  NSF Waterman award DMS-2140043, and a MacArthur Fellowship.

\section{Proof and further overview}\label{sec:proof}
% Proof of Theorem~\ref{thm:main} and further overview of strategy

\subsection{Statements of Main Lemmas}

We state here two lemmas, and motivate the entire argument around these lemmas.  The lemmas are the key novelty.  We also prove Theorem~\ref{thm:main} from the lemmas.

First, we give the geometry-of-numbers lemma.
In $\C^f$, let $B_R:=\{(x_1,\dots,x_f)\,|\, |x_i|\leq R \textrm{ for all $i$}\}$, a polydisc of ``radius'' $R$.
For a subset $S\subset \C^f$, let $U_S:=\{(x_1,\dots,x_f)\in S\,|\, |x_i|=1 \textrm{ for all $i$}\}$, the outermost points on the boundary of the region $B_1 \subset \C^f$.
If we have a lattice $\Lambda$ where $U_\Lambda$ is large, then we can form a point-set in the plane with many unit distances by taking $U_\Lambda\cap B_R$ for some $R$, and then projecting to any coordinate of $\C$.
Since every lattice point in $B_{R-1}$ translated by any point in $U_\Lambda$ is in 
$B_{R}$, if the projection is injective, we obtain at least $\frac{1}{2}|U_\Lambda||\Lambda\cap B_{R-1}|$
unit distance pairs among at most $|\Lambda\cap B_{R}|$ points.  
The following lemma quantifies the size of these sets.

\begin{lemma}\label{L:unit_expansion}
    Let $f$ be a positive integer, and $0<\delta\leq 1$.  Let $\Lambda$ be a full rank lattice in $\C^f$, such that for all non-zero $x\in \Lambda$, at least one coordinate $x_i$ of $x$ has $|x_i|\geq \delta$.  Further assume that the projection of $\Lambda$ onto one of the coordinates of $\C^f$ is injective.  
    Let $v\geq\delta^{-2}\operatorname{covol}(\Lambda)^{1/f}$.  
    
    Suppose that $|U_\Lambda|\geq u^f$ for some $u>0$.
    Then for every $R\geq 2$,
    there exists a translate $a+\Lambda$ of $\Lambda$ such that
    the set $(a+\Lambda)\cap B_R$ projected onto a coordinate gives a point set $\mc{P}$ in the plane with 
    $$
     2\nu(\mc{P})\geq \left(\frac{u \pi R^2}{4v\delta^{2}}\right)^f \textrm{ and } |\mc{P}|\leq\left(\frac{9R^2}{\delta^2}\right)^{f}.
    $$
\end{lemma}
Observe that $9R^2\delta^{-2}>1$, since $R\ge 2$ and $\delta\le 1$.
Thus given $v$, if we can make $u>\frac{36v}{\pi}$, we can choose any $R\ge 2$ and then
\begin{equation}
\label{explicit-general}
\frac{\log(2\nu(\mc{P}))}{\log|\mc{P}|}\geq \frac{\log \frac{u \pi R^2}{4v\delta^{2}}}{\log \frac{9R^2}{\delta^2}}>1.
\end{equation}
If we can make $u>\frac{36v}{\pi}$ and keep $u,v,\delta$ constant while letting $f\rightarrow\infty$ and still finding  lattices satisfying Lemma~\ref{L:unit_expansion}, then we have $|\mc{P}|\rightarrow\infty$
(since $|\nu(\mc{P})|\rightarrow\infty$),
while $\frac{\log(2\nu(P))}{\log|\mc{P}|}$ stays bounded below by a number larger than $1$, proving Theorem~\ref{thm:main}.

\begin{remark}
    Note the above argument does not require making $u$ large in terms of $\delta$.  This is important, as the argument here does not have the flexibility to do that, and indeed produces $\delta^{-1}$ quite large in terms of $u$.
\end{remark}

The parameter $v$ bounds the ``skewness'' of the lattice $\Lambda$.
Taking constant size scalings of Minkowski lattices of algebraic integers in (totally imaginary) number fields of growing degree but bounded root discriminant, which are well-known to exist by work of Golod-Shafarevich, will keep $\delta,v$ constant while $f\rightarrow\infty$.
So the remaining challenge is to produce many elements of these fields with complex absolute value $1$ in every embedding, so that we can make $u$ large compared to $v$.

This will be done by Lemma~\ref{pigeons}. 
Let $e(P)$ denote the ramification index of a prime ideal $P$ in a number field $K$.
For the sake of exposition the following result is more general than necessary,
which may also be useful to readers interested in
modifying or optimizing the overall method.
\begin{lemma}
\label{pigeons}
Let $K$ be a number field embedded in $\C$.
Assume $K = \overline{K}$,
where $\overline{K}$ denotes the complex conjugate of $K$.
Let $P_1,\dots,P_s$ be pairwise distinct prime ideals of $\O_K$
such that $P_i \ne \overline{P}_j$ for all $1\le i,j\le s$.
Let $k_1,\dots,k_s$ be positive integers.
Let $$Q := \prod_{j=1}^{s} (P_j\overline{P}_j)^{k_j}
\subseteq \O_K$$
be an ideal.
Let $U := \{u\in Q^{-2}: |u| = 1\}$.
Then $$|U| \ge \frac{\prod_{j=1}^{s} (k_j+1)}{h(K)}.$$
Moreover, $Q^{-2} \subseteq D^{-1} \O_K$, where
$$D := \prod_{p\mid N(P_1P_2\cdots P_s)}
p^{\max_{j: p\mid N(P_j)} \lceil{2k_j/e(P_j)}\rceil} \in \Z.$$
\end{lemma}

\begin{remark}
Assume $s\ge 1$.
Then Lemma~\ref{pigeons} is vacuous if $K$ is totally real,
and otherwise Dirichlet's unit theorem shows that $|U| = \infty$
unless $K$ is CM. 
%Assume $K$ is neither totally real nor CM.
%Fix a complex (non-real) embedding of $K$,
%and let $F$ be the fixed subfield of $K$ under complex conjugation.
%Then $F\ne K$.
%Thus the map $\O_K^\times\to \O_F^\times$ sending $x$ to %$x\overline{x}$ has positive-rank kernel, because $rank(\O_K^\times) > %rank(\O_F^\times)$ by the unit-defect definition of a CM field.
So in its current form, without including bounds on $u$ in other embeddings of $K$,
Lemma~\ref{pigeons} is useful only if $K$ is CM.
\end{remark}

We will apply Lemma~\ref{pigeons} with $K$ a CM field, with complex conjugation automorphism $c:K\to K$.
Thus all prime ideals $P$ of $\O_K$ above a rational prime that splits completely in $K$ satisfy $P\ne cP$.  Also, above such a rational prime there are many $P_i$ in $K$.
Moreover, an element of $K$ has absolute value $1$
in some embedding if and only if it
has absolute value $1$ in all embeddings, so when $K$ is CM
Lemma~\ref{pigeons} produces elements in $U_\Lambda$ as needed for 
Lemma~\ref{L:unit_expansion}.

To keep $\delta^{-1}$, which will be the $D$ from Lemma~\ref{pigeons},
bounded, we should only take primes $P_i$ above some fixed set of rational primes.  How many and which rational primes should we use?  It turns out not to matter much.
 One idea, due to Will Sawin after he read GPT's proof, is to use primes $P_i$ above a single rational prime.  If we have a sequence of number fields of bounded root discriminant and growing degree, all split completely above at a given rational prime $p$, then we can take all the $k_j$ large compared to that bounded root discriminant to obtain $u$ as desired, and completing the argument.  That such a sequence exists is an immediate consequence of the Golod-Shafarevich Theorem  \cite{GolodShafarevich1964ClassFieldTower,GolodShafarevich1965ClassFieldTowersTranslation} and Shafarevich's relation bounds for Galois groups of class towers \cite{Shafarevich1963RamificationOriginal,Shafarevich1966RamificationTranslation} (see, e.g. Koch \cite[Theorem 11.5]{Koch2002GaloisTheoryPExtensions}).  Indeed, this kind of consequence of these results is widely used, and indeed much stronger results (such as requiring infinitely many primes splitting completely \cite{hajir2021cutting}) are known.

Let $S$ and $T$ be finite disjoint sets of places of $\Q$, with $\infty\in S$.  Let $G_T^S$ be the Galois group of the maximal pro-$2$ extension $\tilde{\Q}$ of $\Q$, unramified outside $T$ and split completely at all primes in $S$, a \emph{class tower}.\footnote{We remark that GPT's original argument used the maximal pro-$3$ extension instead.
The reason for this is unclear, but one line of the CoT mentions that there might be ``nuisances'' for $2$-towers that disappear for $3$-towers.
However, $2$-towers are in fact notationally a bit simpler to work with.}
If $T$ consists of odd primes, for a number field $L\subset \tilde{\Q}$, since only primes in $T$ are ramified and they are ramified tamely, $|\Disc L| \leq \prod_{p\in T}p^{[L:\Q]}$, and so such $L$ have bounded root discriminant (which is, by definition $|\Disc L|^{1/[L:\Q]}$).
Given such an $L,$ we will use $K=L(i)$ as our CM field, and have $|\Disc K|\leq \prod_{p\in T\cup\{2\}}p^{2[L:\Q]}$.  Let $f=\deg(L)$. Then the covolume of $\O_K$ in the Minkowski embedding of $\O_K$ in $\C^f$ is $2^{-f}\sqrt{|\Disc K|}$.
% As a sanity check, the covolume of $\Z[i]$ in $\C$ is $1 = 2^{-1} \sqrt{4}$.

\begin{proof}[Proof of Theorem~\ref{thm:main}]
Let $T$ be any finite set of odd primes with $|T|\geq 6.$
Then $d(G_T^{\{\infty\}})$, which is the minimal number or generators of $G_T^{\{\infty\}}$, is the $2$-rank of the maximal abelian elementary abelian quotient of $G_T^{\{\infty\}}$ (i.e. the Frattini quotient).  This quotient is the Galois group of $L_T$, the maximal totally real multi-quadratic field unramified outside $T$, so by the theory of quadratic fields, $d(G_T^{\{\infty\}})$ is $|T|-1$ unless $T$ contains no primes that are 3 mod $4$, in which case it is $|T|$ (or see \cite[Theorem 11.8]{Koch2002GaloisTheoryPExtensions}).  Then if $S$ is the union of $\{\infty\}$ and a finite set of primes split  completely in $L_T(i)$, the group
$G_T^S$ is the quotient of $G_T^{\{\infty\}}$ by $|S|-1$ relations---the Frobenius elements at the finite places of $S$.  Since these Frobenius elements are trivial in the Frattini quotient, $d(G_T^S)=d(G_T^{\{\infty\}})$, and by \cite[Theorems 11.5,11.8]{Koch2002GaloisTheoryPExtensions}
we have a bound on the relation rank, $$r(G_T^S)\leq r(G_T^{\{\infty\}})+|S|-1 =
d(G_T^{\{\infty\}})+|S|-1.$$ 
As long as
$
|T|-1+|S|-1\leq (|T|-1)^2/4$, then
$r(G_T^S)\leq
d(G_T^S)^2/4,
$
and $G_T^S$ is infinite by the Golod-Shafarevich Theorem \cite{GolodShafarevich1964ClassFieldTower}.
This gives tremendous latitude to pick $T$ and $S$. 
For simplicity, we will take $|S|=2$, where $S=\{p,\infty\}$.  As just one small example, we can let
$T=\{3,5,7,11,13,17\}$ and $S=\{101,\infty\}$.  We have that $L_T=\Q(\sqrt{5},\sqrt{13},\sqrt{17},\sqrt{21},\sqrt{33})$ and that $101$ splits completely in $L_T$  so $d(G_T^S)=5$, and $r(G_T^S)\leq 6$ so 
$G_T^S$ is infinite.  Thus there are number fields $L_j$ of degrees $f_j\rightarrow\infty$ totally real, unramified outside $T$, such that $L_j(i)$ are CM and split completely at (finite) primes in $S$.  We take $K_j=L_j(i).$  Let $r=\prod_{q\in T\cup\{2\}}q $, which is an upper bound for the root discriminant of any $K_j$.

Then we apply Lemma~\ref{pigeons} with all $k_j=k=\lceil 18r^3/\pi\rceil-1$ to these fields in the $f$ pairs of complex conjugate primes $P_i,\bar{P}_i$ above $p.$  
By \cite[p. 143, (7)]{BorelPrasad1989Finiteness}, we have for $[K:\Q]\geq 4,$
$$
h_K\leq |\Disc_K|.
$$
(Thanks to Chat GPT for help finding a convenient reference.  Several different standard arguments can give bounds of the form $h_K\leq c_1|\Disc_K|^{c_2}$ for various constants $c_1,c_2$ and the above is not optimal but suffices for the argument.)
So we have $|U|\geq (k+1)^{f_j} r^{-2f_j}.$  We have $D=p^{2k}$.

Then we apply Lemma~\ref{L:unit_expansion} with $\Lambda=p^{-2k}\O_{K_j}$ in its Minkowski embedding obtain a $U$ with $u=(k+1)r^{-2}$.  Note that since any non-zero element of $\O_{K_j}$ has norm at least 1 and thus has magnitude at least $1$ in some complex embedding, we can take $\delta=p^{-2k}.$  Projection of $p^{-2k}\O_{K_j}$ onto any coordinate in the Minkowski embedding is injective.  
We have that $\covol(p^{-2k}\O_{K_j})=2^{-f_j}\delta^{2f_j} \sqrt{|\Disc K_j|}.$
 Since $\delta^{-2} (2^{-f_j}\delta^{2f_j} \sqrt{|\Disc K_j|})^{1/(f_j)}\leq r/2$, we let $v=r/2$.  Since, $u=\lceil 18r^3/\pi\rceil r^{-2}>36v/\pi$, this proves Theorem~\ref{thm:main} as described after the statement of Lemma~\ref{L:unit_expansion}. 

Explicitly, the exponent in \eqref{explicit-general}
can be taken to be
\begin{equation}
\label{explicit-special}
\sup_{R\ge 2} \frac{\log \frac{u \pi R^2}{4v\delta^{2}}}{\log \frac{9R^2}{\delta^2}}
= 1 + \frac{\log{\frac{u\pi}{36v}}}{\log{\frac{36}{\delta^2}}}\approx 1 + 6.24 \cdot 10^{-38}
\end{equation}
where $v=r/2$,
$\delta\ge 101^{-2\lceil 18r^3/\pi\rceil}$,
and $u=\lceil 18r^3/\pi\rceil r^{-2}$,
with $r = 2\cdot 3\cdot 5\cdot 7\cdot 11\cdot 13\cdot 17$.
\end{proof}

GPT wrote a slightly more difficult argument, taking all $k_j$ to be $1$ and primes $P_i$ above $t$ rational primes.  This requires having an infinite class tower, while controlling the splitting of $t$ rational primes, but keeping the root discriminant small compared to $2^t.$ However, a class tower (with no splitting conditions) has a Galois group of approximately balanced presentation of dimension approximately the number of ramified primes allowed in the tower, and it takes quadratically many relations in terms of generators to make the group finite. Thus, the root discriminant can be a product of $\ell$ primes that we pick, something like $e^{c\ell\log \ell}$, and we can control the splitting of $t$ primes, where $t$ can be taken to be quadratic in $\ell$, and still keep the class tower infinite.  This means we can make $2^t$ plenty large compared to any power of the root discriminant, while having an infinite class tower split at $t$ rational primes.
In retrospect, this AI argument was unnecessarily subtle,
because one can take $t$ arbitrarily large with respect to $\ell$
by Proposition~\ref{HMRmethod1mod4}.

\begin{proposition}
[Cf.~\cite{hajir2021cutting}]
\label{HMRmethod1mod4}
There exists an infinite tower of totally real fields over $\Q$
with bounded root discriminant
and with infinitely many completely split primes $q\equiv 1\bmod{4}$.
\end{proposition}

\begin{proof}
Without the condition $q\equiv 1\bmod{4}$,
this would follow from
\cite[Theorem 2.8 of arXiv version, or Theorem 4 of published version]{hajir2021cutting}
with $K=\Q$, $S=\{3,5,7,11,13,17,\infty\}$, and $p=2$, for instance.
In order to arrange for $q\equiv 1\bmod{4}$,
one can modify the (recursive) proof of \cite[loc. cit.]{hajir2021cutting} as follows:
Every time the Chebotarev density theorem is applied to obtain a split rational prime $q$
of a number field $F$,
we instead apply the Chebotarev density theorem to the number field $F(i)$.
\end{proof}

\subsection{Proofs of Main Lemmas}

\begin{proof}[Proof of Lemma~\ref{L:unit_expansion}]
By averaging over $a\in \C^f/\Lambda$, there exists a choice of $a$ such that
\begin{equation*}
|(a+\Lambda)\cap B_{R-1}|
\ge \left(\frac{\pi (R-1)^2}{\covol(\Lambda)^{1/f}}\right)^f.
\end{equation*}
Fix an injective coordinate projection $\mathcal{P}$
of the bounded window $W := (a+\Lambda)\cap B_R$.
The translation-by-$U_\Lambda$ argument
described before Lemma~\ref{L:unit_expansion} implies that
\begin{equation*}
2\nu(\mathcal{P}) \ge |U_\Lambda||(a+\Lambda)\cap B_{R-1}|
\ge u^f \left(\frac{\pi (R-1)^2}{\covol(\Lambda)^{1/f}}\right)^f
\ge \left(\frac{u \pi R^2}{4v\delta^{2}}\right)^f,
\end{equation*}
since $R-1 \ge R/2$.
% Note: This means (R-1)^2 \ge R^2/4.
On the other hand, $\Lambda$ is $\delta$-separated
in the sup-norm on $\C^f$,
so the sumset (Minkowski sum) of $W$ and $B_{\delta/2}$
has volume $\ge |W| |B_{\delta/2}|$,
and $\le |B_{R+\delta/2}| \le |B_{3R/2}|$.
Thus
\begin{equation*}
|\mathcal{P}| = |W| \le \frac{|B_{3R/2}|}{|B_{\delta/2}|}
= \left(\frac{9R^2}{\delta^2}\right)^{f}. \qedhere
\end{equation*}
%(It might also be possible to give a lower bound on
%$|\mathcal{P}| = |W| = |(a+\Lambda)\cap B_R|$,
%but we do not need to do so.)
\end{proof}

\begin{proof}[Proof of Lemma~\ref{pigeons}]
Use the pigeonhole principle on the classes in $\on{Cl}(K)$
of the ideals $$\prod_{j=1}^{s} P_j^{a_j} \overline{P}_j^{k_j-a_j},$$
where $0\le a_j\le k_j$.
Taking ratios of the ideals landing in the most frequent ideal class, gives
$$\ge \frac{\prod_{j=1}^{s} (k_j+1)}{h(K)}$$
pairwise distinct principal ideals $(\alpha)$
for which $\alpha \overline{\alpha} \in \O_K^\times$.
Now let $u = \alpha / \overline{\alpha} \in Q^{-2}$.
Observe that the ideals $(u) = (\alpha^2)$ are pairwise distinct.
Moreover, $Q^2\mid D \O_K$,
so $Q^2\supseteq D \O_K$,
so $Q^{-2}\subseteq D^{-1} \O_K$.
\end{proof}

\section{Noga Alon}
The Erd\H{o}s unit distance problem \cite{Erdos1946SetsDistances} raised in 1946 
is among the best known open problems
in Combinatorics. It is also arguably the best known problem in 
Discrete Geometry. Indeed, its description in the book of Brass, Moser
and Pach on Research Problems in Discrete Geometry (\cite{BrassMoserPach2005}, Chapter 5)
is:
``The following problem of Erd\H{o}s \cite{Erdos1946SetsDistances} is possibly the best
known (and simplest to explain) problem in combinatorial geometry:
How often can the same distance occur among $n$ points in the plane?''

Let $U(n)$ denote the maximum possible number of unit distances 
determined by $n$ points in the Euclidean plane. Erd\H{o}s
proved that $U(n) \geq n^{1+\Omega(1/\log \log n)}$, and the best
known upper bound is $U(n) \leq O(n^{4/3})$. This was first proved by
Spencer, Szemer\'{e}di and Trotter \cite{SpencerSzemerediTrotter1984UnitDistances} in 1984. Several simpler
proofs of the same bound up to a constant factor 
have been given over the years, the shortest and most elegant one
is due to Sz\'{e}kely \cite{Szekely1997CrossingNumbers}.
More on the rich history of this problem can be found
in \cite{BrassMoserPach2005}. 

The common belief, conjectured by 
Erd\H{o}s, has been that $U(n) \leq n^{1+o(1)}$.
This has been one of Erd\H{o}s' favorite problems, I have heard
him myself mentioning the problem multiple times in his
lectures. 
I believe it would be fair to say that every mathematician working in 
Combinatorial Geometry thought about this 
problem, and lots of mathematicians working in other areas spent
at least some time thinking about it.

Let me also add that although this problem may look at first as 
a recreational one this is not the case,
it is in fact closely related to other 
mathematical areas including Number Theory and Algebraic Geometry.

The solution of the problem by the internal model of Open AI is,
in my opinion, an outstanding achievement, settling a long-standing
open problem. The fact that the correct answer is not $n^{1+o(1)}$
is surprising, and the construction and its analysis apply 
fairly sophisticated tools
from algebraic number theory in an elegant and clever way. 
As explained by the remarks of some of my colleagues here there 
are several reasons that 
explain why AI tools can be better than humans in finding such
a construction. With or without a full agreement with these reasons,
the fact is that the AI was able to do here what lots of excellent
human researchers tried and failed to do.

Like other mathematicians who had the opportunity to  
experiment, even if only briefly in my case, with ChatGPT Pro 5.5, 
my impression has been that AI tools 
are capable of changing research in mathematics in a dramatic way.
The new spectacular solution of the Erd\H{o}s unit distance problem 
convinces me that it is hard to
overestimate the full potential impact of this change.

\section{Thomas Bloom}

This was one of Erd\H{o}s' favourite problems -- he first asked it in 1946 \cite{Erdos1946SetsDistances} and returned to it many times. (The site \url{www.erdosproblems.com}, on which it is Problem \#90, currently lists 14 separate references, and there are no doubt more.) The influential collection of `Research Problems in Discrete Geometry' by Brass, Moser, and Pach \cite{BMP} describes it as `possibly the best known (and simplest to explain) problem in combinatorial geometry'. For an AI to produce a solution to a problem of this calibre is both surprising and impressive.

One rough way to quantify the interest which Erd\H{o}s himself had in his problems is by the prize value attached to them. He first offered a monetary reward in 1982 \cite{Er82e}, of \$300, for a proof or disproof of the upper bound $n^{1+o(1)}$. Interestingly, he wrote in \cite{Er90} that the upper bound of $n^{1+o(1)}$ would be `very difficult to prove' \emph{if true}; certainly Erd\H{o}s had seen many surprising constructions that he had missed at this point, and perhaps saw the possibility that he was also missing something here. The first time the higher prize of \$500 was offered in print appears to be 1995 \cite{Er95}. He was quite clear here that either a proof or disproof of the upper bound $\leq n^{1+o(1)}$ would qualify, so we can unambiguously say that the AI has solved a \$500 Erd\H{o}s problem.

Erd\H{o}s was consistent in his belief that the upper bound of $n^{1+o(1)}$ was true; no doubt he was encouraged as the decades went by without anyone improving on his original lattice based construction. As a result, it is likely that the most of the human efforts spent on this problem have been on trying to prove the upper bound, rather than spending serious time on trying to disprove it.

Perhaps tempting fate, on 16th April I included this problem in a blog post on \url{www.erdosproblems.com} titled (somewhat tongue in cheek) `Top 10 Erd\H{o}s Problems'. I wrote this list as a response to some of the discourse around existing AI solutions to several other, much easier, Erd\H{o}s problems, which had prompted some people to wrongly assume that all of Erd\H{o}s' problems were inconsequential trivialities that had only remained unsolved because nobody had tried to prove them. This is very much not the case -- many of Erd\H{o}s' problems have been intensively studied for decades, and been extremely fruitful in the depth and complexity of the techniques that have been created in attempts to solve them.

The unit distance problem was the only problem from discrete geometry I included; the distinct distance problem was also a strong contender, but since this has been (almost) completely solved by Guth and Katz, the unit distance problem was a better example of a still open yet natural question that has resisted proof for decades.

While I believed that AI would make some progress on at least a couple of the problems in that list eventually, I did not expect this to happen just one month later!

If the result of this paper was a proof of the unit distance problem, that would be truly incredible. While I was still very surprised to hear of the this result, this was dampened slightly when I learnt it was a construction of a counterexample, and still further when I learnt that nature of the construction, being (with the benefit of hindsight) a natural, albeit highly non-trivial, generalisation of the original lattice-based construction of Erd\H{o}s. 

On examining the construction, it becomes more clear how people had missed this before --  it requires the confluence of several different unlikely events: that a good mathematician is
\begin{enumerate}
\item spending significant time in thinking about the unit distance conjecture in the first place;
\item seriously trying to disprove it, despite the oft-repeated belief of Erd\H{o}s that it is true;
\item believes that there is mileage in generalising the original construction to other number fields, and so is willing to expend significant time in exploring such constructions\footnote{See the remarks of Will Sawin and Jacob Tsimerman in this paper for a discussion of how one might have started along such a path, but dismissed such an approach as `this can never work'.}; and
\item sufficiently familiar with the relevant parts of class field theory to recognise that the appropriately phrased question about infinite towers of number fields with appropriate parameters can be solved using existing theory.
\end{enumerate}
The AI met all of these criteria, and its success here echoes previous achievements: it often produces the most surprising results by persevering down paths that a human may have dismissed as not worth their time to explore, combining superhuman levels of patience with familiarity with a vast array of technical machinery.

When assessing the importance and influence of an AI-generated proof, a question I ask myself is: has this taught us something new about the problem? Do we understand discrete geometry better now? I think the answer is a moderated yes: this shows that there is a lot more that number theoretic constructions have to say about these sorts of questions than we suspected; moreover, that the number theory required can be very deep. No doubt many algebraic number theorists will be taking a close look at other open problems in discrete geometry in the coming months.

On the other hand, perhaps some in the area will be a little disappointed with how little this tells us: it does not introduce any powerful new geometric tools, or hitherto unsuspected structural results, that a \emph{proof} of the unit distance conjecture would likely have called for. Still, while perhaps not the proof of a conjecture that we had hoped for, no doubt this construction and the ideas involved will have a major impact in discrete geometry.

One aspect of this proof should not be overlooked: while the original proof produced by AI was completely valid, it was significantly improved by the human researchers at OpenAI and the many other mathematicians involved in the present paper. The human still plays a vital role in discussing, digesting, and improving this proof, and exploring its consequences.

The frontiers of knowledge are very spiky, and no doubt the coming months and years will see similar successes in many other areas of mathematics, where long-standing open problems are resolved by an AI revealing unexpected connections and pushing the existing technical machinery to its limit. AI is helping us to more fully explore the cathedral of mathematics we have build over the centuries; what other unseen wonders are waiting in the wings?
\section{W T Gowers}

I do not have the background in algebraic number theory to make a detailed assessment of the disproof of Erd\H os's unit-distance conjecture, so instead I shall make some tentative comments about what it tells us about the current capabilities of AI. 

My experience of learning about the solution was an interesting one. I heard about it from S\'ebastien Bubeck during a Zoom call, but misunderstood what he was saying and thought that the model had proved an upper bound of $n^{1+o(1)}$. The Zoom call took place in the late afternoon and I spent the evening adjusting my world view: if AI could come up with a proof like that, then maybe it would be all over for mathematicians very soon. The next morning I and the other authors of this paper received an email about the result, and only then did I understand that it had disproved the conjecture rather than proving it, which came as a big relief. It is interesting to reflect on why it should have been a relief, since either way this is the first example of a famous (in my mathematical circles at least) open problem being solved by AI with no human intervention once it had been trained and then given the problem to solve.

The short answer is that, without knowing anything about the solution, I could more easily imagine a model coming up with a counterexample while still lacking some essential mathematical capabilities than imagine it coming up with a proof. The breakthrough of Larry Guth and Nets Katz that solved the closely related Erd\H os distance problem introduced new and unexpected tools into combinatorial geometry that led to the solutions of many further problems. However, the unit distance conjecture continued to resist (for reasons we now understand!), so appeared to require a further major idea. It was exactly that -- a major new idea -- that would have been disturbing. A counterexample, on the other hand, was something one could imagine a computer coming up with by trying lots of things and at some point getting lucky, without needing ``deep insight''. To be clear, I am not saying that that is what actually happened -- just that it was possible to imagine that it had happened. 

Now that I have seen the solution and seen some of the reactions to it by people who understand it in detail, I find myself not only trying to assess what AI has achieved in this particular case, but also thinking more generally about how such assessments can possibly be made. Can we still identify some mathematical capability that human mathematicians have and AI does not yet have? If so, what might that capability be, and how could one go about demonstrating that AI still lacks it? Almost certainly the answer to the first question will have to be quantitative rather than qualitative. That is, we are unlikely to be able to show that there is something we can do that current AI models cannot in principle do at all, but we might be able to show that there are things we can still do much more efficiently than those models. But when a model has just solved a major open problem, it is clear that even a modest conclusion like that will not be straightforward to demonstrate, and indeed isn't obviously true. 

I run a group in automatic theorem proving at Cambridge, and one of our activities has been to try to understand what it is for a mathematics problem to be difficult. An informal idea that has come up in discussions with Jacob Loader, Marco dos Santos and Anand Tadipatri is something that one might call ``Kolmogorov complexity modulo experts''. The rough idea is to consider the difficulty of a proof to be the length of the shortest sequence of bits that would provide experts with enough hints to reconstruct the proof. It isn't quite as simple as that, since an expert can always do without a $k$-bit hint by doing a brute-force search through all $k$-bit sequences. However, one can get round that objection by adjusting the measure to something like $2^k+t$, where $k$ is the length of the ``hint sequence'' and $t$ is the time needed to reconstruct the proof. Even this is not perfect, since it ignores the time it takes to explore the tree of potential hint sequences, but one can adjust for that too. Another objection is that some hints may be short but very surprising: there are proofs that take a long time to be discovered because although they are simple, they involve ideas that have ``no right to work'' or that come from very different areas of mathematics. Again, one can modify the measure to take this phenomenon into account. 

For the purposes of this discussion I shall assume that such difficulties have been ironed out, and I shall refer to the ``complexity'' of a proof to be some suitably developed notion of Kolmogorov complexity modulo experts. The following questions then arise naturally: are human mathematicians still able to find proofs with a larger complexity than AI is capable of finding, and what is the complexity of the solution that has just been discovered to the Erd\H os unit distance problem? 

For the moment, I'd like to imagine a hint sequence as something like a multi-part question on a question sheet, designed to help a suitably expert mathematician work through the proof by reducing it to a sequence of exercises. The first part of such a hint sequence might be something like this.
\bigskip

\centerline{(i) Look for a counterexample}
\bigskip

\noindent When assessing the ``length'' of this hint, one should imagine how it might be efficiently encoded. For this purpose, one could imagine having the help of a ``hint book'', in which bit sequences are translated into hints. The hints in such a book would be rather generic, and the book would be designed in such a way that hints that are frequently useful would have short encodings. In such a book, the two length-1 sequences (0) and (1) might might well encode ``look for a counterexample'' and ``look for a proof'', respectively, in which case this would be a one-bit hint. However, this example also demonstrates the need to take surprisal into account. For instance, I have myself thought a small amount about the unit-distance problem, but it never occurred to me to try to disprove it. Somehow I was too convinced by a story that turned out to be incorrect: that the distinct-distances conjecture follows from the unit-distance conjecture, which doesn't have an obvious counterexample, and the distinct-distance conjecture has now been proved, so it is likely that eventually the unit-distance conjecture will be proved too. As a result, my subjective probability that the conjecture was false was quite low, which I believe was the commonly held view, making the hint particularly useful. Thus, the value of the one bit of information should be taken as greater than 1 in order to reflect the considerable updating of one's perception of the problem once one has obtained it.

A second hint might be this.
\bigskip

\centerline{(ii) Take the best known construction and generalize it.}
\bigskip

\noindent In the hypothetical hint book, highly generic and frequently applied moves like these would presumably have very short encodings. I also think this hint would be sufficient for a mathematician to make significant inroads into the proof. The standard construction, due to Erd\H os, is a square grid of size $\sqrt n\times\sqrt n$. Since one is talking about distances between points, it becomes very natural to think of the elements of this grid as Gaussian integers, and once one does that, one has an algebraic structure that can be generalized. 

To extract the most possible from this example, one looks for an integer $m\leq n$ that can be written as a sum of two squares in as many ways as possible, since this will give rise to many distances equal to $\sqrt m$ (and afterwards one can of course normalize by dividing by $\sqrt m$). This one can do by taking $m$ to have many prime factors congruent to 1 mod 4, and the result is a lower bound for the unit-distance problem of $n^{1+c/\log\log n}$. 

We can look at the construction as follows: the circle of radius $\sqrt m$ contains a large set $U$ of Gaussian integers, and a $\sqrt n\times\sqrt n$ grid of Gaussian integers has many differences that belong to $U$. Without the need for a further hint, it is natural to try to generalize this by taking another ring of integers that resembles the Gaussian integers, and that leads quite quickly to the idea that ChatGPT sets out: to replace the ring $\mathbb Z$ by the ring of integers in a totally real field $L$ and to replace $\mathbb Q(i)$ by $L(i)$. 

At this stage in the proof-discovery process, one needs to ``drop a level'' and start considering more detailed calculations about the properties that $L$ will need to have. Here I must leave it to somebody else to suggest where the difficulties would arise for an expert who has decided to pursue this line of attack, and what a minimal set of hints might be that would enable such an expert to complete the proof. That is not very easy to assess because of the time/hint tradeoff mentioned earlier, not to mention a multitude of psychological factors such as the degree of motivation of the expert in question. However, I asked Will Sawin, who suggested to me that the following hint would have been helpful to make substantial further progress, though probably not sufficient to reduce generating the rest of the proof to an exercise. (Of course, that depends on what counts as an exercise, which is itself not very easy to make precise.) 
\bigskip

\centerline{(iii) Try a sequence of number fields of increasing degree,} 
\centerline{but work with prime ideals of bounded norm.}
\bigskip

\noindent As written, this hint is not so generic, but in context it is more so, since by this stage number fields are already part of the picture, which should allow for considerable compression. For example, maybe the first half of the hint could be something like, ``Replace what you were previously trying to prove by a non-uniform version.'' 

Even if more hints are still needed, my impression is that the total length, even taking surprisal into account, is fairly short. Many of the ideas needed for the proof were present in the literature already, and for such ideas either no hint is needed, since the expert is aware of that piece of literature, or a highly generic ``look it up" hint would be enough. 

Clearly I am speculating wildly, but my hunch is that the hint sequence that would have been necessary to guide an expert to Guth and Katz's proof of the Erd\H os distinct-distances conjecture would have been quite a bit longer, because that proof involved several surprising ideas with non-obvious connections to the problem. If that hunch is correct, it would provide some kind of justification for my relief that this result was a disproof rather than a proof. Perhaps what we are seeing at the moment is not that AI is about to overtake human mathematicians, but rather that there are certain styles of problem where it has a distinct advantage. It has an encyclopaedic knowledge of mathematics, and it does not have to worry nearly as much as we do about time management, so it is good at finding surprising connections, and it can afford to try quite hard to prove statements that seem unlikely to be true -- provided, in both cases, that the complexity of the proofs it finds is not too high. 

Is there any reason to suppose that a conclusion like this is correct? More to the point, if it is correct now, is there any reason to suppose that it will remain correct in, say, two years' time? The only way I can see that happening is if AI is for some fundamental reason exploring ``hint space'' in a much less efficient way than we do. I don't see any obvious signs of that from its chain of thought, but it could perhaps be that a lot of additional ``actual thought'' is going on behind each step of what it presents as its chain of thought. My current bet is that progress in AI mathematics is \emph{not} about to reach a plateau, and that we will soon see AI solutions to many problems that we will find hard to explain away as easier than expected with hindsight. (Here I am assuming that there will at least be the usual kinds of efficiency gains, where what a large model can do this year, a much smaller one can do next year. Without that, solving lots of problems might be too expensive, not to mention environmentally unfriendly.)

In any case, there is no doubt that the solution to the unit-distance problem is a milestone in AI mathematics: if a human had written the paper and submitted it to the Annals of Mathematics and I had been asked for a quick opinion, I would have recommended acceptance without any hesitation. No previous AI-generated proof has come close to that. Furthermore, even if it is correct that AI cannot yet find a proof that needs a long hint sequence, such proofs are very difficult to find for humans as well, so in the unlikely event that progress in AI mathematics does suddenly stall, we have still probably entered an era where it will become very difficult for humans to compete with AI at solving mathematical problems.\footnote{I chose my words carefully there, since solving problems is not all that mathematicians do. My guess is that AI will soon reach a high level at other activities such as building theories, formulating definitions and asking interesting questions, but that is a separate discussion.} 

\section{Daniel Litt}
Before hearing about its solution, I had little experience with the unit distance problem—I had seen it on erdosproblems.com and was aware of the statement but knew little about the history of the problem. After an internal model at OpenAI produced a solution, I was asked to check its correctness by Mark Sellke and Mehtaab Sawhney at OpenAI, with the idea that I had some familiarity with the algebraic number theory being used, particularly around Golod-Shafarevich’s construction of infinite class field towers. It did not take long for me to convince myself that the solution was correct, not to mention quite clever and natural. I spent a fair amount of time the following weekend nerd-sniped into thinking about various possible optimizations and variants, though I think these have now been subsumed by the ideas explained by others in this paper. 

This is the first example of a result produced autonomously by an AI that I find exciting in itself, as opposed to as a leading indicator. That said, the mathematical context is quite far from my expertise and I will leave others better-qualified than me to comment further on the solution and its interest. Instead, I will speculate briefly about what it tells us about the human practice of mathematics.

There are a few examples of relatively well-known open problems resolved via a fairly short, clever argument: famously, the finite field Kakeya conjecture, proven by Dvir; the sensitivity conjecture, proven by Huang; and a few others. Arguably, this solution to the unit distance problem has the same flavor. My sense is that such examples are historically relatively rare, though I suspect we are about to find out about quite a few more. What explains the existence of such problems? How common are they? It seems to me that the near-term future of research mathematics revolves to some extent around the answers to these questions.

One possible explanation for such low-hanging fruit: those working on the problem have anchored onto a non-optimal approach or belief about the truth (for example, in this case, that Erd\H{o}s’s conjecture on the topic was true). Another: the solution requires ideas from areas with which most of those working on the problem are unfamiliar. These explanations, if correct, should cause us some discomfort. They suggest that incentives towards specialization and silo-ing, though understandable, have cost us some high-quality science.

In my own area---algebraic and arithmetic geometry---there are simply very few practitioners, and so arguably all problems are attention-bottlenecked. On the other hand, my sense is that it is unusual for well-known open problems to have answers accessible to clever argument, as opposed to new theory. Many of us work on \emph{programs}, rather than \emph{problems} (though I personally count myself as a problem-solver).  It will be interesting to see how AI impacts these areas; the impact thus far has been minimal, though I expect this state of affairs will not last long.

Finally, it is illuminating to contrast the most productive current approach to doing mathematics by AI to the way humans mathematicians work. At any given time a human will, driven by their personal curiosity, choose a small number of questions and try to understand them deeply. By contrast, the best autonomous AI mathematics has been produced by trawling through entire problem lists and solving some portion of the listed problems. This is a vast expansion of the attention aimed at mathematical problems, and perhaps will serve to better focus future human attention and curiosity.

\section{Will Sawin}

I find it difficult to summarize my overall impressions, which would in any case overlap with the points raised by other mathematicians, so instead I will make two somewhat disconnected observations:

Let me first explain a subtlety that likely made it more difficult for mathematicians working on the problem to discover this argument: Erd\H{o}s's original lower bound, using a set of lattice points, may be interpreted as taking points of absolute value bounded by a large parameter in the ring of integers $\mathbb Z[i]$ of the imaginary quadratic field $\mathbb Q(i)$. 

In trying to generalize this, the most natural approach is to take points of absolute value bounded by a large parameter in the ring of integers $\mathcal O_K$ of some fixed CM field $K$. On the other hand, OpenAI's internal model's approach was to take a set of points of absolute value bounded by a fixed parameter in the ring of integers $\mathcal O_K$ of a CM field $K$ of increasing degree. Getting from the first approach to the second approach would be much more intuitive if the bounds obtained from the first approach grew with the degree of the field $K$. However, they do not, for the following reason:

 If one normalizes the set so that unit distances come from elements of $\mathcal O_F$ of norm an element $\alpha$ whose ideal $(\alpha)$ is a product of small split primes, which for simplicity we will take to be $\prod_{P \in S} P$ for some set $S$ of primes of $\mathcal O_F$, then the number of such $\alpha$ will be $2^{ \#S}$ times a factor coming from the class group and we will need to take the absolute value to be bounded by $\left(\prod_{P \in S} NP\right)^{1/{2d}}$ which means the set of points has size proportional to $\prod_{P \in S} NP$ , where $NP = \# (\mathcal O_K/P)$ is the norm of $P$. (We could take $k$'th powers of some primes for $k>1$, but this does not significantly affect the asymptotics in this setting.) We want to maximize $\prod_{P \in S} P$ while minimizing $\prod_{P \in S} NP$  while ensuring all primes $P$ in $S$ split in $K$, which is done by taking $S$ to be the set of prime ideals of $\mathcal O_F$ split in $\mathcal O_K$ that have norm $<X$. 

Having done this, the prime number theorem for prime ideals gives that $\# S = (1+o(1)) \frac{X }{2\log X} $ and $\prod_{P \in S} P= e^{ (1+o(1) X}$, which exactly recovers Erd\H{o}s's lower bound. Thus, there is no indication that the choice of field matters. One could examine the lower-order terms in the prime number theorem, but for $X$ large the dominant terms come from the zeroes of the Dedekind zeta function of $K$, and these terms, of size $X^{1/2}$, are greater than the terms arising from the class number. There is no apparent reason one field would be better-behaved than another field for this, as each zero gives oscillatory contributions that give better bounds for some $X$  and worse bounds for other $X$, and thus no reason to try a sequence of fields.

Let me now explain the potential to generalize this approach to two other, related questions in combinatorial geometry: The distinct distances problem, and the unit distances problem in dimension $3$. In both cases, there are serious obstacles to any possible generalization.

The distinct distances problem asks for the minimum number of distinct distances between a set of $n$ distinct points in the plane. The distinct distance problem and the unit distance problem are related, in that by the pigeonhole principle, a set of $n$ points which contains pairs of points at only $m$ distances must have some at least $n(n-1)/m$ pairs of points of distance $d$ for at least one $d$, and then dividing by $d$ we get a set of points with $n(n-1)/m$ pairs of points at unit distance. However, it is much harder to upper bound the number of distinct distances than to lower bound the number of pairs of points at unit distance. In fact, Guth and Katz proved~\cite{GuthKatz2015DistinctDistances} a lower bound of the form $n/\log n$ for the number of distinct distances represented by $n$ points n the plane, which differs from Erd\H{o}s's upper bound by a facto of only $\sqrt{\log n}$.

There is a natural approach to the distinct distance problem using number fields. Given a CM field $K$ with totally real subfield $F$ of degree $d$, we can take a set of points in $\mathcal O_K$ whose absolute value, in each complex embedding, is at most $R$. One expects roughly $ (\pi R^2)^d / \sqrt{\Delta_K}$ such points, and the squared distance between any two such points is a totally positive element of $\mathcal O_F$ whose absolute value in each real embedding, is at most $4R^2$, and one expects the number of elements of $\mathcal O_F$ with these properties is roughly $(4R^2)^d/ \sqrt{\Delta_F}$. However, not all elements of $\mathcal O_F$ can be the squared length: Just as a positive integer cannot be the squared length of a vector in the standard lattice, i.e. a sum of two squares, if it is divisible, with odd multiplicity, by a prime congruent to $3$ modulo $4$, an element of $\mathcal O_F$ cannot be the squared length of an element of $\mathcal O_K$ if it is divisible, with odd multiplicity, by a prime ideal of $\mathcal O_F$ inert in $\mathcal O_K$. For each prime $P$ of $\mathcal O_F$, the proportion of elements of $\mathcal O_F$ which $P$ divides to odd multiplicity is $\frac{1}{ 1+ NP^{-1}}$.

Hence, if we construct $\mathcal O_F$ to have a set $S$ of small primes that remain inert in $\mathcal O_K$, and we assume the primes behave independently, we will get a number of distinct distances that is roughly $\frac{ (4R^2)^d}{\sqrt{\Delta_F}} \prod_{P \in S} \frac{1}{1+NP^{-1}}$. To get a nontrivial bound, the factor $\prod_{P\in S} (1+ NP^{-1})$ must be better than $(4/\pi)^d \frac{\sqrt{\Delta_K}}{\sqrt{\Delta_F}}$. This is much more difficult than the unit distance problem, since the quantity to beat is similar, growing exponentially in $d$ for fields of bounded root discriminant, but the quantity that must beat it is much smaller, getting a factor of $1+ NP^{-1}$ for each prime instead of $P$. If we could somehow ensure that $k/\log k$ primes of $\mathbb Q$ split completely in $F$ and are inert in $K$, the best case scenario is that these are the first $k/\log k$ primes, which by the prime number theorem are roughly the primes $<k$. Then $\prod_{P\in S} (1+ NP^{-1})$  would be $\prod_{p < k} (1+p^{-1})^d = e^{ d (\log \log k+O(1))}$ by Mertens' theorem. However, to obtain these split primes by Golod-Shafarevich, we would need $\sqrt{k/\log k}$ ramified primes, which would make the discriminant exponentially large in $d \sqrt{k/\log k}$, and thus much larger than $e^{ d (\log \log k+O(1))}$.

The unit distance problem in $\mathbb R^3$ asks for the maximum number of pairs of points at unit distance in a set of points in $\mathbb R^3$. The points in a ball of radius $n^{1/3}$ in $\mathbb Z^{3}$ have squared distances $O(n^{2/3})$ and thus some squared distance must arise from at least $n \cdot n/ n^{2/3} =n^{4/3}$ pairs of points. More precisely, the number of times a given squared distance $m$ arises is proportional to the number of representations of $m$ as a sum of three squares. It follows from the class number formula that the number of representations of $m$ as a sum of three squares is $\sqrt{m}$ times a bounded value times $L(1, \chi_m)$. Walfisz~\cite{walfisz1942class} showed that $L(1,\chi_m)$ can be larger than $\log \log m$ infinitely often (which is sharp under the generalized Riemann hypothesis, as shown by Littlewood~\cite{Littlewood1928}). It follows that some squared distance can arise from at least a constant times $n^{4/3} \log \log n$ pairs of points.

Again, there is a natural approach by number fields. Given a totally real field $F$ of degree $d$ and a positive definite quadratic form $Q$ in three variables over $\mathcal O_F$, one can embed $\mathcal O_F^3$ into $\mathbb R^3$ in such a way that the squared length of a vector $v\in \mathcal O_F^3$ is given by $Q(v)$. Taking the set of vectors in $\mathcal O_F^3$ of length bounded by $R$ in each real embedding, one expects roughly $R^{3d}( \sqrt{N\Delta_Q} \sqrt{\Delta_F}^3)$ such vectors, where we drop factors that are exponential in $d$ independent of $F$ for simplicity. The number of possible squared distances obtained from pairs of such vectors is  $(R^{2d}/ \sqrt{\Delta_F}$, so a typical squared distance is obtained $ R^{4d} / ( N \Delta_Q \sqrt{\Delta_F}^5)$ times. However, the number of pairs of vectors with squared distance $\alpha$ is proportional to the number of representations of $\alpha$ by the quadratic form $Q$. Now Siegel's mass formula shows that, when we average over quadratic forms $Q$ with a given discriminant, the average number of representations is an exponential factor in $d$ times $ R^{4d} L(1, \chi_{\alpha \Delta_Q}) / ( (N \Delta_Q)^{1/3} \sqrt{\Delta_F}^5)$.

For this to be better than $n^{4/3}$, we must have $L(1,\chi_{\alpha \Delta_Q})$ greater than $(N \Delta_Q )^{1/3} \sqrt{\Delta_F}$. We may express $L(1,\chi_{\alpha \Delta_Q})$ as a conditionally convergent Euler product with a factor of $\frac{1}{1- \frac{1}{NP}}$ for each prime $P$ of $\mathcal O_F$ split in $F(\sqrt{\alpha \Delta_Q})$ and a factor of $\frac{1}{1+\frac{1}{NP}}$ for each prime $P$ of $\mathcal O_F$ inert in $F(\sqrt{\alpha \Delta_Q})$, so the situation is similar to the distinct distances problem, except worse as we also have to contend with the factors at the inert primes. 

\section{Arul Shankar}

I had not encountered this problem before seeing the proof from OpenAI, and I found the proof to be a clean execution of a very beautiful idea and quite well written up.
This impressive proof is (loosely speaking) a generalization of Erdos' original lower bound. In his proof, Erdos took elements in $\Z[i]$ of bounded absolute value, and used these to construct algebraic numbers of absolute value $1$ with bounded denominators. Then the bound goes to infinity, essentially by expanding the set of integer primes allowed to contribute to the elements in $\Z[i]$. This new proof changes this paradigm quite a bit. First, instead of fixing the field and varying the integer primes, the set of integer primes is fixed and the field is varied. Second, rather than taking a family of fixed degree fields to vary over, the new proof uses a class field tower of fields.

The first change of perspective is familiar in some fields of number theory, especially arithmetic statistics. For example, it is common to study a complementary question to the Chebotarev density theorem, where instead of fixing a field and counting the number of (bounded norm) primes that split completely, we fix a prime and count the number of fields (of bounded height) in a family in which the prime splits completely. The second paradigm shift is significantly less common, especially, as pointed out previously, in "fixed dimension" settings. All the same, I would consider this to be a very "human" proof, though a extremely ingenious one.

The model's CoT is deeply interesting. It is noteworthy that a significant majority of the thoughts are trying to construct a counterexample to the widely believed upper bound, rather than trying to prove it. This argues that the model has some combination of good intuition, willingness to try approaches considered long-shot by the community, and a predisposition to attempt constructions.

The CoT showed the model trying out a vast array of ideas from a wide range of mathematics for the required construction. The model went through ideas pretty quickly, but when it reached the crucial idea (in the paragraph starting with "Suppose optimistically that..."), it honed in on the proof quite methodically.

In my opinion this paper demonstrates that current AI models go beyond just helpers to human mathematicians -- they are capable of having original ingenious ideas, and then carrying them out to fruition.

\section{Jacob Tsimerman}

This is a really impressive piece of work, and I would accept it for any journal without hesitation. I actually briefly worked on this problem and tried to make a counterexample, but failed to make progress.

On Boris Alexeev's suggestion, I thought about this problem with the idea of making a counterexample stemming from a varying family of bounded degree number fields. Increasing degree occured to me, but is a very scary dynamic and often doesn't work out. Moreover, it is hard to think through the analytic regimes and retain guiding intuition - it consumes much time and frequently doesn't work out. While it's true in the final solution that nothing is all that surprising, there are many ways to attempt to set this construction up (how big are the primes?  How big is the ball? Do you take large products?  how much splitting does one insist on - this is a tradeoff with how easy it is to make the field). It is definitely an intimidating construction to see through even if you know what is going on, and even harder to go play for yourself. It's always tempting to look at a completed proof and declare it obvious after the fact. 

This may indicate one way that AI systems have an edge: it's not just that they can try all known methods, but they can play for longer and in more treacherous waters than mathematicians  without getting overwhelmed. Of course this is not yet robustly true, but this may be a foreshadowing event.

\section{Victor Wang}

My previous exposure to this topic came from some beautiful lectures over ten years ago by Larry Guth, on the distinct distances problem, and by Po-Shen Loh, on the unit distance problem.
They focused on the geometric and combinatorial aspects of these problems.
I have not thought much about the problems since, but recently I enjoyed hearing about \cite{Alon2025Typical} from Noga Alon.
Only now have I begun to really appreciate the role number theory has to play for special metrics, such as the present $\ell^2$-norm on $\R^2$, and more generally the $\ell^k$-norm on $\R^s$ where $s,k\ge 2$ are integers.

The case $s=k$ is particularly interesting to me.
Here the unit distance problem, specialized to scalar multiples of integer boxes $[-B,B]^k$, is essentially equivalent to the study of the Hardy-Littlewood Hypothesis~K for $k$-fold sums of $k$th powers of nonnegative integers.
If $r_k(n)$ denotes the number of ways to write an integer $n\ge 1$ as such a sum of powers, then $r_2(n) \le O(n^{o(1)})$, but this is no longer true for $r_3(n)$ (Mahler), and it represents a tantalizing open question for $k\ge 4$;
see \url{https://www.erdosproblems.com/322} and \url{http://thomasbloom.org/notes/hypothesisk.html} for details.
Meanwhile, the distinct distances problem, specialized to $[-B,B]^k$, is essentially equivalent to asking how many integers of size $O(B^k)$ can be written as a sum of $k$ nonnegative $k$th powers.
This is a difficult open question for $k\ge 3$, lacking the multiplicative structure present for $k=2$;
see the recent survey \cite{maynard2026sums} for details in the case $k=3$.

Sequences of objects of increasing degree have proven decisive in many problems over finite fields (and function fields) in general, and over number fields in Iwasawa theory and elsewhere.
Their role in combinatorics also has precedent.
The new unit-distance construction is a reminder of the power of this idea, making use of high-dimensional multiplicative number theory in high-degree fields.

The fact that a complete argument can be given in a few pages (assuming background in algebraic number theory at the level of a first or second semester graduate course), as in the present remarks, made the verification process relatively smooth.
Had the final digested argument been a bit longer, the process may have been trickier, as participants may have been less willing to invest time in checking details.
It will be interesting to see how formalization progresses alongside AI.

The implicit social contract between mathematicians and AI companies
deserves further attention.
When Hajir, Maire, and Ramakrishna wrote their beautiful papers
\cite{hajir2001asymptotically,hajir2021cutting},
did they have in mind that an AI might eventually use their work
(as the CoT likely indicates)
to derive headline results,
potentially with significant ensuing financial implications?
When we make our work freely available on the arXiv,
do we all implicitly want it to be freely available to AI as well?

I do not want to comment further on the trajectory of AI, which seems to me to be a complicated question involving physics, materials, society, and the environment.

\section{Melanie Matchett Wood}

I had not heard of this problem before hearing of the solution from Open AI.  I find the argument to be a beautiful application of number theory to a natural, concrete question.

It is easy to jump to hasty conclusions, but what we can learn about humans, AI, and mathematics from this development is somewhat subtle.  I believe if the level and type of human expertise that is represented on this note had been assembled to find a counterexample to this conjecture a month ago, and those people put in similar amounts of time working on it than they did to reading and thinking about Chat GPT’s solution, the mathematicians would have found a counterexample.  However, without the claimed proof by Chat GPT, there is no particular reason anyone would have tried to look for a counterexample, assembled a group of experts with the appropriate expertise, or that the experts would have agreed to turn their attention to this problem.  We can all be reminded by this development of how frequently interesting and powerful things happen mathematically when one applies ideas from one field to another, and think about how AI can help us find more cross-field applications.

This result does not show us all the times AI has claimed to have a proof of something and been wrong.  Without that context (which many of us have just from personal experience), it is also easy to draw incorrect conclusions about the current state of AI and research mathematics.  In many cases, it will be  easier for AI to convince humans it has a proof than to come up with a correct mathematical argument, and I believe that we as mathematicians are not sufficiently prepared for this.  

One other concern that directly arises in this development is that there is a history of closely related ideas in the literature, some of which are mentioned above, but which are not appropriately referenced  in Chat GPT’s paper.  If a human came up with this argument and didn’t cite such previous work, we would assume that they were unfamiliar with the previous work and came up with the ideas independently, since our professional norms require us to cite previous work whose ideas influenced our work.  On the other hand, Chat GPT is in some sense “familiar” with all the previous work.  In the future we can expect humans to write many papers that include ideas suggested by AI.  Mathematicians need to think about what best practices and proper citation is in these kind of situations, and come to a common understanding as a community.  

Properly contextualized, we can see from this and other developments that AI will play an increasingly important role in research mathematics.   As a mathematics community, we urgently need to plan for how we can keep our work rigorous and correct, properly acknowledge the influence of previous ideas, and preserve a high level of human understanding of mathematics as we move forward in our use of AI as part of process of research mathematics.

\bibliographystyle{amsplain0}
\bibliography{main}

\providecommand{\bysame}{\leavevmode\hbox to3em{\hrulefill}\thinspace}
\providecommand{\MR}{\relax\ifhmode\unskip\space\fi MR }
\providecommand{\MRhref}[2]{%
  \href{http://www.ams.org/mathscinet-getitem?mr=#1}{#2}
}
\providecommand{\href}[2]{#2}
\begin{thebibliography}{10}

\bibitem{akhtari2025effective}
Shabnam Akhtari, Jeffrey~D. Vaaler, and Martin Widmer, \emph{Effective equidistribution of norm one elements in {CM-fields}}, arXiv preprint arXiv:2507.10387, to appear in Annali della Scuola Normale Superiore di Pisa, Classe di Scienze (2025).

\bibitem{Alon2025Typical}
Noga Alon, Matija Buci\'c, and Lisa Sauermann, \emph{Unit and distinct distances in typical norms}, Geom. Funct. Anal. \textbf{35} (2025), 1--42.

\bibitem{ApplebyFlammiaMcConnellYard2020RayClassFields}
Marcus Appleby, Steven Flammia, Gary McConnell, and Jon Yard, \emph{Generating ray class fields of real quadratic fields via complex equiangular lines}, Acta Arithmetica \textbf{192} (2020), 211--233.

\bibitem{Batyrev1998Manin}
Victor~V. Batyrev and Yuri Tschinkel, \emph{Manin's conjecture for toric varieties}, J. Algebraic Geom. \textbf{7} (1998), 15--53.

\bibitem{belolipetsky2012manifolds}
Mikhail Belolipetsky and Alexander Lubotzky, \emph{Manifolds counting and class field towers}, Advances in Mathematics \textbf{229} (2012), 3123--3146.

\bibitem{BloomErdosProblem90}
Thomas~F. Bloom, \emph{Erd{\H{o}}s problem {\#}90}, \url{https://www.erdosproblems.com/90}, 2026, Accessed 2026-05-09.

\bibitem{BorelPrasad1989Finiteness}
Armand Borel and Gopal Prasad, \emph{Finiteness theorems for discrete subgroups of bounded covolume in semi-simple groups}, Publications Math{'e}matiques de l'IH{'E}S \textbf{69} (1989), 119--171.

\bibitem{BMP}
Peter Brass, William Moser, and J\'{a}nos Pach, \emph{Research problems in discrete geometry}, Springer, New York, 2005.

\bibitem{BrassMoserPach2005}
Peter Brass, William O.~J. Moser, and J{\'a}nos Pach, \emph{Research problems in discrete geometry}, Springer, New York, 2005.

\bibitem{EllenbergVenkatesh2007Reflection}
Jordan~S. Ellenberg and Akshay Venkatesh, \emph{Reflection principles and bounds for class group torsion}, International Mathematics Research Notices \textbf{2007} (2007), Art. ID rnm002, 18 pp.

\bibitem{Er82e}
Paul Erd\H{o}s, \emph{Some of my favourite problems which recently have been solved}, Proceedings of the International {M}athematical {C}onference, Singapore 1981 (Singapore, 1981), North-Holland Math. Stud., vol.~74, North-Holland, Amsterdam-New York, 1982, pp.~59--79.

\bibitem{Er90}
Paul Erd\H{o}s, \emph{Some of my favourite unsolved problems}, A tribute to Paul Erd\H{o}s (1990), 467--478.

\bibitem{Er95}
Paul Erd\H{o}s, \emph{Some of my favourite problems in number theory, combinatorics, and geometry}, Resenhas \textbf{2} (1995), 165--186, Combinatorics Week (Portuguese) (S\~Ao Paulo, 1994).

\bibitem{Erdos1946SetsDistances}
Paul Erd{\H{o}}s, \emph{On sets of distances of {$n$} points}, American Mathematical Monthly \textbf{53} (1946), 248--250.

\bibitem{GolodShafarevich1964ClassFieldTower}
E.~S. Golod and I.~R. Shafarevich, \emph{On the class field tower}, Izv. Akad. Nauk SSSR Ser. Mat. \textbf{28} (1964), 261--272 (Russian), English translation: Amer. Math. Soc. Transl. (2) 48 (1965), 91--102.

\bibitem{GolodShafarevich1965ClassFieldTowersTranslation}
E.~S. Golod and I.~R. Shafarevich, \emph{On class field towers}, Fourteen Papers on Logic, Algebra, Complex Variables and Topology, American Mathematical Society Translations, Series 2, vol.~48, American Mathematical Society, Providence, RI, 1965, English translation of the 1964 Russian paper, pp.~91--102.

\bibitem{Guruswami2003Constructions}
Venkatesan Guruswami, \emph{Constructions of codes from number fields}, IEEE Transactions on Information Theory \textbf{49} (2003), 594--603.

\bibitem{GuthKatz2015DistinctDistances}
Larry Guth and Nets~Hawk Katz, \emph{On the {E}rd{\H{o}}s distinct distances problem in the plane}, Annals of Mathematics \textbf{181} (2015), 155--190.

\bibitem{hajir2001asymptotically}
Farshid Hajir and Christian Maire, \emph{Asymptotically good towers of global fields}, European Congress of Mathematics, Vol. II (Barcelona, 2000) (Carles Casacuberta, Rosa~Maria Mir{\'o}-Roig, Joan Verdera, and Sebasti{\`a} Xamb{\'o}-Descamps, eds.), Progress in Mathematics, vol. 202, Birkh{\"a}user, Basel, 2001, pp.~207--218.

\bibitem{hajir2021cutting}
Farshid Hajir, Christian Maire, and Ravi Ramakrishna, \emph{Cutting towers of number fields}, Annales Math{\'e}matiques du Qu{\'e}bec \textbf{45} (2021), 321--345.

\bibitem{Koch2002GaloisTheoryPExtensions}
Helmut Koch, \emph{Galois theory of {$p$}-extensions}, Springer Monographs in Mathematics, Springer, Berlin, Heidelberg, 2002.

\bibitem{Kopp2021SICPOVMStark}
Gene~S. Kopp, \emph{{SIC-POVMs} and the {Stark} conjectures}, International Mathematics Research Notices (2021), 13812--13838.

\bibitem{Lenstra1986Codes}
Jr. Lenstra, H.~W., \emph{Codes from algebraic number fields}, Mathematics and Computer Science, II, CWI Monographs, vol.~4, North-Holland Publishing Co., Amsterdam, 1986, Proceedings of the conference held in Amsterdam, 1986, pp.~95--104.

\bibitem{Liang2024QaryGV}
Xue-Bin Liang, \emph{The $q$-ary gilbert-varshamov bound can be improved for all but finitely many positive integers $q$}, 2024, Preprint.

\bibitem{LitsynTsfasman1987Constructive}
S.~N. Litsyn and M.~A. Tsfasman, \emph{Constructive high-dimensional sphere packings}, Duke Mathematical Journal \textbf{54} (1987), 147--161.

\bibitem{Littlewood1928}
J.~E. Littlewood, \emph{On the class-number of the corpus $p(\sqrt{-k})$}, Proceedings of the London Mathematical Society \textbf{s2-27} (1928), 358–372.

\bibitem{LuzziVehkalahtiLing2018AlmostUniversal}
Laura Luzzi, Roope Vehkalahti, and Cong Ling, \emph{Almost universal codes for {MIMO} wiretap channels}, IEEE Transactions on Information Theory \textbf{64} (2018), 7218--7241.

\bibitem{Maire2018}
Christian Maire and Fr{\'e}d{\'e}rique Oggier, \emph{Maximal order codes over number fields}, Journal of Pure and Applied Algebra \textbf{222} (2018), 1827--1858.

\bibitem{maynard2026sums}
James Maynard, \emph{Sums of three positive cubes}, Journal of the London Mathematical Society \textbf{113} (2026), e70554.

\bibitem{Shafarevich1963RamificationOriginal}
Igor~R. Shafarevich, \emph{Extensions {\`a} points de ramification donn{\'e}s (en russe)}, Publications Math{\'e}matiques de l'IH{\'E}S \textbf{18} (1963), 71--92 (Russian).

\bibitem{Shafarevich1966RamificationTranslation}
Igor~R. Shafarevich, \emph{Extensions with given points of ramification}, American Mathematical Society Translations, Series 2 \textbf{59} (1966), 128--149, English translation by J. W. S. Cassels; title also cited as ``Extensions with prescribed ramification points''.

\bibitem{SpencerSzemerediTrotter1984UnitDistances}
Joel~H. Spencer, Endre Szemer{\'e}di, and William~T. Trotter, \emph{Unit distances in the euclidean plane}, Graph Theory and Combinatorics (B{\'e}la Bollob{\'a}s, ed.), Academic Press, London, 1984, Proceedings of the Cambridge Conference, 1983, pp.~293--303.

\bibitem{Szekely1997CrossingNumbers}
L{\'a}szl{\'o}~A. Sz{\'e}kely, \emph{Crossing numbers and hard {E}rd{\H{o}}s problems in discrete geometry}, Combinatorics, Probability and Computing \textbf{6} (1997), 353--358.

\bibitem{Tsfasman1991GlobalFields}
Michael~A. Tsfasman, \emph{Global fields, codes and sphere packings}, Ast{\'e}risque \textbf{198--200} (1991), 373--396.

\bibitem{VehkalahtiLuzzi2015NumberField}
Roope Vehkalahti and Laura Luzzi, \emph{Number field lattices achieve gaussian and rayleigh channel capacity within a constant gap}, 2015 IEEE International Symposium on Information Theory (ISIT), IEEE, 2015, pp.~436--440.

\bibitem{Venkatesh2013SpherePackings}
Akshay Venkatesh, \emph{A note on sphere packings in high dimension}, International Mathematics Research Notices (2013), 1628--1642.

\bibitem{walfisz1942class}
Walfisz, \emph{On the class-number of binary quadratic forms}, Trav. Inst. Math. Tbilissi \textbf{11} (1942), 57--71.

\end{thebibliography}

\end{document}